\documentclass[11pt]{amsart}
\usepackage{amsmath, amssymb, amscd, mathrsfs, url, pinlabel,verbatim}
\usepackage[pagebackref]{hyperref}
\usepackage[margin=1.25in,marginparwidth=1in,centering,letterpaper,dvips]{geometry}
\usepackage{color,dcpic,latexsym,graphicx,epstopdf,comment}
\usepackage[all]{xy}
\usepackage[dvipsnames]{xcolor}
\usepackage{tikz,tikz-cd,pgfplots}
\usepackage{upquote,tabularx,textcomp}
\usepackage[shortlabels]{enumitem}
\usepackage[color=blue!20!white,textsize=tiny]{todonotes}
\usepackage{longtable}

\title{Torus knots, the A-polynomial, and $\SL(2,\C)$}

\author[John A. Baldwin]{John A. Baldwin}
\address{Department of Mathematics \\ Boston College}
\email{john.baldwin@bc.edu}

\author[Steven Sivek]{Steven Sivek}
\address{Department of Mathematics\\Imperial College London}
\email{s.sivek@imperial.ac.uk}

\makeatletter
\newtheorem*{rep@theorem}{\rep@title}
\newcommand{\newreptheorem}[2]{%
\newenvironment{rep#1}[1]{%
 \def\rep@title{#2 \ref{##1}}%
 \begin{rep@theorem}}%
 {\end{rep@theorem}}}
\makeatother

\newtheorem {theorem}{Theorem}
\newreptheorem{theorem}{Theorem}
\newtheorem {lemma}[theorem]{Lemma}
\newtheorem {proposition}[theorem]{Proposition}
\newtheorem {corollary}[theorem]{Corollary}

\newtheorem {question}[theorem]{Question}

\numberwithin{equation}{section}
\numberwithin{theorem}{section}

\theoremstyle{definition}
\newtheorem{definition}[theorem]{Definition}

\newtheorem{remark}[theorem]{Remark}
\newtheorem*{remark*}{Remark}

\newtheorem{example}[theorem]{Example}

\newlist{pcases}{enumerate}{1}
\setlist[pcases]{
  label=\bf{Case~\arabic*:}\protect\thiscase.~,
  ref=\arabic*,
  align=left,
  labelsep=0pt,
  leftmargin=0pt,
  labelwidth=0pt,
  parsep=0pt
}
\newcommand{\case}[1][]{%
  \if\relax\detokenize{#1}\relax
    \def\thiscase{}%
  \else
    \def\thiscase{~#1}%
  \fi
  \item
}

\newcommand{\Z}{\mathbb{Z}}

\newcommand{\R}{\mathbb{R}}
\newcommand{\C}{\mathbb{C}}

\newcommand{\Q}{\mathbb{Q}}

\newcommand{\cN}{\mathcal{N}}

\newcommand\SU{\mathrm{SU}}
\newcommand\SL{\mathrm{SL}}
\newcommand\PSL{\mathrm{PSL}}

\newcommand\KHI{\mathit{KHI}}

\newcommand\Is{I^\#}

\DeclareFontFamily{U}{mathx}{\hyphenchar\font45}
\DeclareFontShape{U}{mathx}{m}{n}{
      <5> <6> <7> <8> <9> <10>
      <10.95> <12> <14.4> <17.28> <20.74> <24.88>
      mathx10
      }{}
\DeclareSymbolFont{mathx}{U}{mathx}{m}{n}
\DeclareFontSubstitution{U}{mathx}{m}{n}
\DeclareMathAccent{\widecheck}{0}{mathx}{"71}

\newcommand{\ab}{\operatorname{ab}}

\newcommand{\tr}{\operatorname{tr}}

\newcommand{\pt}{\mathrm{pt}}

\newcommand{\mirror}[1]{\overline{#1}}

\makeatletter
\DeclareFontFamily{OMX}{MnSymbolE}{}
\DeclareSymbolFont{MnLargeSymbols}{OMX}{MnSymbolE}{m}{n}
\SetSymbolFont{MnLargeSymbols}{bold}{OMX}{MnSymbolE}{b}{n}
\DeclareFontShape{OMX}{MnSymbolE}{m}{n}{
    <-6>  MnSymbolE5
   <6-7>  MnSymbolE6
   <7-8>  MnSymbolE7
   <8-9>  MnSymbolE8
   <9-10> MnSymbolE9
  <10-12> MnSymbolE10
  <12->   MnSymbolE12
}{}
\DeclareFontShape{OMX}{MnSymbolE}{b}{n}{
    <-6>  MnSymbolE-Bold5
   <6-7>  MnSymbolE-Bold6
   <7-8>  MnSymbolE-Bold7
   <8-9>  MnSymbolE-Bold8
   <9-10> MnSymbolE-Bold9
  <10-12> MnSymbolE-Bold10
  <12->   MnSymbolE-Bold12
}{}

\let\llangle\@undefined
\let\rrangle\@undefined
\DeclareMathDelimiter{\llangle}{\mathopen}%
                     {MnLargeSymbols}{'164}{MnLargeSymbols}{'164}
\DeclareMathDelimiter{\rrangle}{\mathclose}%
                     {MnLargeSymbols}{'171}{MnLargeSymbols}{'171}
\makeatother

\newcounter{desccount}

\newcommand{\descref}[1]{\hyperref[#1]{#1}}

\usetikzlibrary{calc,intersections}
\tikzset{every picture/.style=thick}
\tikzset{link/.style = { white, double = black, line width = 1.75pt, double distance = 1.25pt, looseness=1.75 }}
\tikzset{crossing/.style = {draw, circle, dotted, minimum size=0.5cm, inner sep=0, outer sep=0}}
\pgfplotsset{compat=1.12}

\begin{document}


\begin{abstract}
The A-polynomial of a knot is defined in terms of $\SL(2,\C)$ representations of the knot group, and encodes information about essential surfaces in the knot complement. In 2005, Dunfield--Garoufalidis and Boyer--Zhang proved that it detects the unknot using Kronheimer--Mrowka's work on the Property P conjecture. Here we use more recent results from instanton Floer homology to prove that a version of the A-polynomial detects whether a knot is a torus knot. We moreover completely determine which individual torus knots are detected by this A-polynomial.  These results enable progress towards a folklore conjecture about boundary slopes of non-torus knots.  Finally, we use similar ideas to prove that a knot in the 3-sphere admits infinitely many $\SL(2,\C)$-abelian Dehn surgeries if and only if it is a torus knot, affirming a variant of a conjecture due to Sivek--Zentner.
\end{abstract}

\maketitle

\section{Introduction}\label{sec:intro}

One of the most  fruitful ways to  probe the geometry and topology  of a 3-manifold $Y$ is to study the  space of homomorphisms from its fundamental group  to another group $G$, \[R_G(Y) = \mathrm{Hom}(\pi_1(Y),G).\]
However $R_G(Y)$ sometimes fails to capture  any more  information about $Y$ than does its first homology.  It is interesting to study this failure, which prompts the definition: a 3-manifold $Y$ is called \emph{$G$-abelian} if every homomorphism in $R_G(Y)$ has abelian image. These can be viewed as the simplest 3-manifolds from the perspective of the representation variety.

For example, a longstanding conjecture \cite[Problem 3.105]{kirby-list} which has received  recent attention \cite{lpcz,bs-stein} holds that $S^3$ is the only $\SU(2)$-abelian  homology 3-sphere. Zentner proved this  with $\SL(2,\C)$ in place of $\SU(2)$ in \cite{zentner}. There has also been considerable  interest in studying  $\SU(2)$-abelian Dehn surgeries on knots, as in  Kronheimer and Mrowka's work on the Property P conjecture and its offshoots
\cite{km-p,km-su2,bs-lspace,blsy}.

One of our two main results concerns a conjecture by Sivek and Zentner on this theme \cite{sz-pillowcase}, which asserts that  torus knots are the only  knots in $S^3$ with  infinitely many $\SU(2)$-abelian Dehn surgeries. We prove  this with $\SL(2,\C)$ in place of $\SU(2)$:

\begin{theorem}
\label{thm:SLabeliantorus}
Let  $K\subset S^3$ be a knot. Then $S^3_r(K)$ is $\SL(2,\C)$-abelian for infinitely many $r\in \Q$ if and only if $K$ is a torus knot.
\end{theorem}

This is ultimately a theorem about $\SL(2,\C)$-representation varieties of knot complements, as is  our next result  on the A-polynomial of Cooper, Culler, Gillet, Long, and Shalen.  Recall that the \emph{A-polynomial} of a knot $K \subset S^3$ \cite{ccgls} is a  polynomial $A_K(M,L) \in \Z[M,L]$ that describes the pairs of values taken by representations
\[ \pi_1(S^3\setminus K) \to \SL(2,\C) \]
on a meridian and longitude of $K$.  The representation variety always contains a component of abelian representations, contributing a factor of $L-1$; discarding this component, we get the \emph{enhanced A-polynomial}
\[ \tilde{A}_K(M,L) \in \Z[M,L], \] which is related to the original by \begin{equation}\label{eqn:run}A_K(M,L) = \operatorname{lcm}(\tilde{A}_K(M,L), L-1),\end{equation}
where $\operatorname{lcm}$ denotes the least common multiple.

The Newton polygon of $\tilde{A}_K(M,L)$ is a convex polygon $\cN(K)\subset\R^2$, and we say that $K$ is \emph{thin} if this polygon is contained in a line. Dunfield and Garoufalidis \cite{dunfield-garoufalidis} and Boyer and Zhang \cite{boyer-zhang} independently proved that $K$ is the unknot if and only if $\cN(K)$ is a point. Our other main result is a classification of all remaining thin knots:

\begin{theorem} \label{thm:thin-torus}
A knot $K \subset S^3$ is thin if and only if it is a torus knot.
\end{theorem}

That is, $\tilde{A}_K(M,L)$ distinguishes torus knots from all non-torus knots. In addition, Theorem \ref{thm:thin-torus} enables us to completely determine which specific torus knots can be individually detected by the enhanced A-polynomial (in particular, it detects all $T_{2,2n+1}$ torus knots):

\begin{corollary}
\label{cor:detectionA}
 $\tilde{A}_K(M,L)$ detects the nontrivial torus knot $T_{a,b}$ if and only if:
\begin{itemize}
\item $|a|=2$ or $|b|=2$, or 
\item $|a|$ and $|b|$ are powers of different primes.\end{itemize} 
\end{corollary}

To clarify the difference between these results, we remark that, for example, if \[\tilde{A}_K(M,L) = \tilde{A}_{T_{3,35}}(M,L)\] then Theorem~\ref{thm:thin-torus} says that $K$ must be a torus knot.  According to Corollary~\ref{cor:detectionA}, however, we cannot say specifically which torus knot $K$ is in this case; in fact, $K$ could also be either $T_{5,21}$ or $T_{7,15}$, both of which have the same enhanced A-polynomial as $T_{3,35}$. By contrast, if \[\tilde{A}_K(M,L) = \tilde{A}_{T_{3,25}}(M,L)\] then Corollary \ref{cor:detectionA} says that $K = T_{3,25}$ as both 3 and 25 are prime powers.

Corollary~\ref{cor:detectionA} greatly improves upon our previous work in \cite[Corollary~10.16]{bs-lspace}, which showed that the enhanced A-polynomial detects a much more limited set of torus knots.
Theorem \ref{thm:thin-torus} further implies that the combination of the enhanced A-polynomial and the degree of the symmetrized Alexander polynomial detects \emph{every} torus knot:

\begin{corollary}
\label{cor:detectionAdeg}
 The pair $\big(\tilde{A}_K(M,L), \deg \Delta_K(t)\big)$ detects every torus knot. 
\end{corollary}

This is a substantial strengthening of results by Ni and Zhang, who showed in \cite{ni-zhang} that the combination of the enhanced A-polynomial and \emph{knot Floer homology} detects each torus knot. Among other things, they use the fact that knot Floer homology detects fibered knots to argue that a knot with the same combination of these invariants as a torus knot is fibered. By contrast, we know directly that thin knots are fibered via our work on instanton L-space surgeries \cite{bs-lspace}. We only use the degree of the Alexander polynomial (which is encoded by but is a much weaker invariant than knot Floer homology) to distinguish torus knots $T_{a,b} \neq T_{c,d}$ with $ab=cd$, which $\tilde{A}_K(M,L)$ cannot generally do by itself, per Corollary~\ref{cor:detectionA}.

While Theorem \ref{thm:thin-torus} shows that the \emph{enhanced} A-polynomial distinguishes torus knots from all other knots, it does not quite prove that the  A-polynomial does. Indeed, if $A_{K}(M,L) = A_{T_{a,b}}(M,L)$ then \eqref{eqn:run} implies that either 
\[\tilde{A}_K(M,L) = \tilde{A}_{T_{a,b}}(M,L) \,\textrm{ or }\, \tilde{A}_K(M,L) = (L-1)\tilde{A}_{T_{a,b}}(M,L).\]
In the first case, Theorem \ref{thm:thin-torus} implies that $K$ is a torus knot. The  problem is that  we cannot rule out the second case at present. We therefore pose the following:

\begin{question}
\label{ques:unreduced}
Does $A_{K}(M,L)$ distinguish torus knots from all other knots?
\end{question}

One of the most powerful features of the A-polynomial is its relationship with boundary slopes of incompressible surfaces in the knot complement. We describe below an application of Theorem \ref{thm:thin-torus} along these lines, after a bit of background and motivation.

Recall that a \emph{boundary slope} for a knot is an isotopy class of curves  on the boundary of a tubular neighborhood of the knot arising as the intersection of the boundary torus with an incompressible surface in the knot complement. It was shown in \cite{ccgls} that the slope of each side of the Newton polygon of $A_K$ is a boundary slope for $K$.

For instance, a nontrivial torus knot $T_{a,b}$ has two boundary slopes: the cabling slope $ab$ and the Seifert slope 0. This is reflected in the fact that $\cN(T_{a,b})$ is a line segment of slope $ab$, which implies that the Newton polygon of its A-polynomial is a parallelogram  with sides of slopes $ab$ and $0$, where the slope $0$ sides come from the  extra  $L-1$ factor. This example motivates the following strengthening of Question \ref{ques:unreduced}:

\begin{question}
\label{ques:parallelogram} Is it true that $K$ is a nontrivial torus knot if and only if the Newton polygon of $A_K(M,L)$ is a parallelogram?
\end{question}

A longstanding folklore conjecture asserts that nontrivial torus knots are in fact the \emph{only} knots in $S^3$ with two boundary slopes. This would follow from an affirmative answer to Question \ref{ques:parallelogram}.
With Theorem \ref{thm:thin-torus}, we are able to make the following partial progress:

\begin{theorem}
\label{thm:slopes}
If $K\subset S^3$ is a nontrivial fibered knot but not a torus knot, then either:
\begin{itemize}
\item there is a connected, separating, incompressible surface in the complement of $K$ which meets the boundary torus in curves of slope 0, or
\item $K$ has at least 3 boundary slopes.
\end{itemize}
\end{theorem}
We note that the first possibility in Theorem~\ref{thm:slopes} certainly can occur: for example, if $K$ is a knot of genus at least $2$ and $S^3_0(K)$ is toroidal, then the incompressible torus in $S^3_0(K)$ must be separating by \cite[Corollary~8.3]{gabai-foliations3}, so its intersection with $S^3 \setminus N(K)$ is such a surface.  As examples we can take the pretzel knots $K = P(2,n,-n)$ for any odd $n\geq 3$, as these are fibered of genus $n-1$ \cite[Theorem~6.7, Case~2B]{gabai-fibred} and have toroidal zero-surgery \cite[Theorem~1.1(2)]{wu}.

Finally, we note that the A-polynomial is defined for knots in arbitrary 3-manifolds, and ask whether our main theorems have analogues in this more general setting:

\begin{question}
Do analogues of Theorems \ref{thm:SLabeliantorus} and  \ref{thm:thin-torus} hold for knots in other 3-manifolds?
\end{question}

Suppose, for instance, that  $K\subset Y$ is a knot with irreducible complement.  Is it true that $K$ is thin, or admits infinitely many $\SL(2,\C)$-abelian surgeries, if and only if its complement is Seifert fibered over the disk with at most two singular fibers?

\subsection{Proof sketch}

The proofs of Theorems \ref{thm:SLabeliantorus} and \ref{thm:thin-torus} are similar in spirit; we hint here at our proof of the latter. Very briefly, suppose that $K\subset S^3$ is a nontrivial thin knot. Then $K$ is not hyperbolic \cite{ccgls}, and Sivek--Zentner prove in \cite{sz-pillowcase} that $K$ admits infinitely many $\SU(2)$-abelian surgeries (Proposition \ref{prop:thin-basics}). It then follows from a combination of Sivek--Zentner's work and our own in \cite{bs-stein,bs-lspace} that $K$ is an instanton L-space knot and hence fibered (Theorems \ref{thm:bs-lspace} and \ref{thm:sz-averse}). Moreover, our work in \cite{bs-trefoil}, combined with newer work of Li--Ye \cite{li-ye-surgery} on the surgery formula in framed instanton homology, places strong restrictions on the coefficients of the Alexander polynomial of $K$ (Corollary~\ref{cor:lspace-top-two}). We use the above results to prove that $K$ is not a satellite knot: if it were, then we show that there would be another thin satellite knot whose pattern embeds in $S^3$ as a nontrivial torus knot, but we prove that any such knot would violate the Alexander polynomial constraints (Proposition~\ref{prop:torus-satellite-non-lspace}). Therefore, since $K$ is neither hyperbolic nor a satellite, $K$ must be a torus knot.

Even though our methods largely involve instanton gauge theory, we are unable to prove Theorems~\ref{thm:SLabeliantorus} and \ref{thm:thin-torus} with $\SU(2)$ in place of $\SL(2,\C)$.  The key obstacle to doing so is that many hyperbolic knots (such as the $(-2,3,7)$ pretzel) are instanton L-space knots.  In such cases instanton homology doesn't say anything about whether large surgeries on these knots are $\SU(2)$-abelian, but we do know that large surgeries are hyperbolic \cite{thurston-notes} and so we at least get faithful $\SL(2,\C)$ representations \cite{cs-splittings}.  (See Proposition~\ref{prop:thin-basics} and Lemma~\ref{lem:sl2-averse-hyperbolic}.)  If we knew the $\SU(2)$ version of each theorem for hyperbolic knots, then our arguments for satellite knots would apply verbatim with $\SU(2)$ in place of $\SL(2,\C)$.

\subsection{Organization}
In \S\ref{sec:prelim}, we provide a review of instanton L-spaces, $\SU(2)$-averse knots, and the A-polynomial, and prove some  preliminary results. In \S\ref{sec:satellites}, we prove  new results about instanton L-space knots which are  satellites. In \S\ref{sec:thin}, we apply these results  to prove Theorem \ref{thm:thin-torus} and Corollaries \ref{cor:detectionA}, \ref{cor:detectionAdeg}, and Theorem \ref{thm:slopes}. We also prove as a consequence of Theorem \ref{thm:thin-torus} that only torus knots are $\PSL(2,\C)$-averse. In \S\ref{sec:SL}, we prove Theorem \ref{thm:SLabeliantorus}.

\subsection{Acknowledgments} We thank Nathan Dunfield and Xingru Zhang for helpful correspondence, and the referee for helpful comments on the initial version of this paper. The first author was supported by NSF CAREER Grant DMS-1454865.

\section{Preliminaries}\label{sec:prelim}

In this section we provide  background  and establish some preliminary results which will be important in our proofs of both Theorems \ref{thm:SLabeliantorus} and \ref{thm:thin-torus}. Here and throughout, we adopt the standard convention that the unknot is a torus knot but  not a satellite knot. Moreover, when writing that $K=P(C)$ is a satellite knot with pattern $P\subset S^1\times D^2$ and companion $C\subset S^3$, we assume that $C$ is nontrivial and that $P$ is not a core of the solid torus.

\subsection{Instanton L-space knots} Recall that a  3-manifold $Y$ is called an \emph{instanton L-space} if it is a rational homology 3-sphere and its \emph{framed instanton homology} $\Is(Y)$,  defined by Kronheimer and Mrowka in \cite{km-yaft}, satisfies \[\operatorname{rank}\Is(Y) = |H_1(Y)|.\] A knot $K\subset S^3$ is called an \emph{instanton L-space knot} if the Dehn surgery  $S^3_r(K)$ is an instanton L-space for some positive $r\in\Q$. We proved the following in {\cite[Theorem~1.15]{bs-lspace}}:

\begin{theorem}[] \label{thm:bs-lspace}
If $K \subset S^3$ is a nontrivial instanton L-space knot then:
\begin{itemize}
\item $K$ is fibered and strongly quasipositive, and 
\item $S^3_r(K)$ is an instanton L-space  if and only if $r \geq 2g(K)-1$.
\end{itemize}
\end{theorem}

\noindent Further restrictions on instanton L-space knots come from  \emph{instanton knot homology}
\[ \KHI(K) = \bigoplus_{i=-g(K)}^{g(K)} \KHI(K,i) \]
defined in \cite{km-excision}, whose graded Euler characteristic agrees (up to sign) with the Alexander polynomial, \[\sum_i \chi\big(\KHI(K,i)\big)\cdot t^i=\pm \Delta_K(t),\]  by \cite{km-alexander,lim}.  (Here and throughout the paper we always use $\Delta_K(t)$ to mean the \emph{symmetrized} Alexander polynomial, satisfying $\Delta_K(t) = \Delta_K(t^{-1})$ and $\Delta_K(1)=1$.)  In particular, Li and Ye proved the following in {\cite[Theorem~1.9]{li-ye-surgery}}:

\begin{theorem}[] \label{thm:ly-lspace}
If $K \subset S^3$ is an instanton L-space knot then 
\[ \operatorname{rank} \KHI(K,j) = 0\text{ or }1 \]
for each $j$.  The coefficients of the Alexander polynomial $\Delta_K(t)$ therefore all lie in $\{-1,0,1\}$. Moreover, the nonzero coefficients alternate in sign.
\end{theorem}

Combined with our previous work in \cite{bs-trefoil}, this yields the corollary below:

\begin{corollary} \label{cor:lspace-top-two}
If $K\subset S^3$ is an instanton L-space knot of genus $g \geq 1$, then the Alexander polynomial of $K$ has leading terms 
\[ \pm \Delta_K(t) = t^g - t^{g-1} + \ldots, \]
where the omitted terms all have degree at most $g-2$.
\end{corollary}

\begin{proof}
Theorem~\ref{thm:bs-lspace} says that $K$ is fibered, from which it follows that \[ \operatorname{rank} \KHI(K,g) = 1 \]
by \cite[Corollary~7.19]{km-excision}.  Since $K$ is a nontrivial fibered knot, we also know from \cite[Theorem~1.7]{bs-trefoil} that
\[ \operatorname{rank} \KHI(K,g-1) \geq 1, \]
 in which case Theorem~\ref{thm:ly-lspace} says that \[\operatorname{rank} \KHI(K,g-1)=1,\] and so the $t^{g-1}$-coefficient of $\Delta_K(t)$ is $\pm1$.  Since the nonzero coefficients of $\Delta_K(t)$ alternate in sign, the signs of the $t^g$- and $t^{g-1}$-coefficients must be $+1$ and $-1$ in some order.
\end{proof}

We will be particularly interested in instanton L-space knots which are satellites, and we study these in more detail in \S\ref{sec:satellites}. Here, we establish the following preliminary lemma:

\begin{lemma} \label{lem:satellite-lspace}
Let $K = P(C)$ be a  satellite knot with pattern $P$ and companion $C$, and let $w \geq 0$ denote the winding number of $P$.  If $K$ is an instanton L-space knot, then:
\begin{itemize}
\item $w \geq 1$,
\item $P(U)$ is not the unknot,
\item $P(U)$ and $C$ are both fibered.
\end{itemize}
\end{lemma}

\begin{proof}
First note that $K$ is fibered by Theorem~\ref{thm:bs-lspace}. This implies that $w\neq 0$ and that both $P(U)$ and $C$ are fibered by a combination of \cite[Corollary~4.15]{burde-zieschang} and Stallings' theorem that a knot is fibered if and only if the commutator subgroup of its knot group is finitely generated \cite[Theorem~5.1]{burde-zieschang}.  It remains to prove that $P(U)$ is nontrivial.

Suppose for a contradiction that $P(U)$ is the unknot, and let $g=g(K)$.  Then
\[ \Delta_K(t) = \Delta_{P(U)}(t) \cdot \Delta_C(t^w) = \Delta_C(t^w). \]
The left side has the form $\pm(t^g - t^{g-1} + \dots)$ by Corollary~\ref{cor:lspace-top-two}, while the right side can only have that form if $w=1$.  On the other hand, Hirasawa, Murasugi, and Silver \cite[Corollary~1]{hms} proved that if $w=1$ and  $P(U)=U$, and if $K$ is fibered, then $P$ is isotopic to the core of the solid torus, in which case $K$ is not a satellite knot, a contradiction.
\end{proof}

\begin{remark}
It should be true that if $K = P(C)$ is an instanton L-space knot then  $P(U)$ and $C$ are as well.  The Heegaard Floer analogue of this claim is a theorem of Hanselman, Rasmussen, and Watson \cite[Theorem~1.15]{hrw}, but the instanton version remains open.
\end{remark}

\subsection{$\SU(2)$-averse knots}

Given a group $G$, recall from the introduction that a 3-manifold $Y$ is said to be  \emph{$G$-abelian} if every homomorphism \[\pi_1(Y)\to G\] has abelian image. We will be interested in knots with infinitely many $G$-abelian  surgeries, for $G = \SU(2)$ and $\SL(2,\C)$, leading to the  shorthand:

\begin{definition}
A  knot $K\subset S^3$ is \emph{$G$-averse} if it is nontrivial and there are infinitely many rational numbers $r$ for which  $S^3_r(K)$ is $G$-abelian.
\end{definition}

Sivek and Zentner proved the following in \cite[Theorem~1.1]{sz-pillowcase}, tying together the notions of $\SU(2)$-averse knots and instanton L-spaces:

\begin{theorem}[] \label{thm:sz-averse}
If $K \subset S^3$ is $\SU(2)$-averse then the set of $\SU(2)$-abelian slopes is bounded and  has a single accumulation point $r(K)$. Moreover: 
\begin{itemize}
\item  $r(K)$  is rational and satisfies $|r(K)|>2$, and
\item if  $r(K) > 0$ then $S^3_s(K)$ is an instanton L-space for all $s \geq \lceil r(K) \rceil - 1$.
\end{itemize}
\end{theorem}

For example, every nontrivial torus knot $T_{a,b}$ is $\SU(2)$-averse with $r(T_{a,b}) = ab$, because $(nab+1)/n$-surgery on $T_{a,b}$ is a lens space,  and hence $\SU(2)$-abelian, for all $n\in\Z$.  Sivek and Zentner conjecture  \cite{sz-pillowcase} that these are in fact the \emph{only} $\SU(2)$-averse knots. Our Theorem \ref{thm:SLabeliantorus} is the statement that their conjecture holds with $\SL(2,\C)$ in place of $\SU(2)$.

Note that if $K$ is $\SU(2)$-averse then so is its mirror $\mirror{K}$, and  $r(\mirror{K}) = -r(K)$. In particular, Theorem \ref{thm:sz-averse} implies the following, which  leads to strong restrictions on $\SU(2)$-averse knots by the results in the previous section:

\begin{corollary}
\label{cor:SUL}
If $K \subset S^3$ is $\SU(2)$-averse then  $K$ or $\mirror{K}$ is an instanton L-space knot.
\end{corollary} 

We will use this in combination with the results of the previous section to study $\SL(2,\C)$-averse knots in \S\ref{sec:SL}, noting that a knot which is $\SL(2,\C)$-averse is also $\SU(2)$-averse.

\subsection{The A-polynomial}

The A-polynomial $A_K(M,L)$ of a knot $K \subset S^3$, introduced by Cooper, Culler, Gillet, Long, and Shalen \cite{ccgls}, is defined in terms of  representations
\[ \rho: \pi_1(S^3\setminus \nu(K)) \to \SL(2,\C) \]
that restrict to diagonal matrices on the peripheral subgroup, meaning that
\begin{align*}
\rho(\mu) &= \begin{pmatrix} M & 0 \\ 0 & M^{-1} \end{pmatrix}, &
\rho(\lambda) &= \begin{pmatrix} L & 0 \\ 0 & L^{-1} \end{pmatrix},
\end{align*} where $\mu$ and $\lambda$ are a meridian and longitude on the boundary of the tubular neighborhood $\nu(K)$ of the knot.
The tuples $(M,L)$ which arise in this way define a subset of $\C^\ast \times \C^\ast$. The 1-dimensional components of the Zariski closure of this subset form an affine plane curve $V(K)$, and we take \[A_K(M,L) \in \Z[M,L] \] to be the defining polynomial of $V(K)$. It is normalized to have integer coefficients whose greatest common divisor is 1, and no repeated factors, and is thus well-defined up to sign.

One of the components $X_\mathrm{red}$ of $V(K)$ corresponds to the abelian representations
\[ \rho_M: \pi_1(S^3\setminus \nu(K)) \twoheadrightarrow H_1(S^3\setminus \nu(K)) \to \SL(2,\C), \qquad M \in \C^\ast \]
that send a meridian $\mu$ to the matrix $\left(\begin{smallmatrix} M & 0 \\ 0 & M^{-1}\end{smallmatrix}\right)$.  Then $\rho_M(\lambda)$ is always the identity matrix since $\lambda$ is nullhomologous, so the image of this component in $\C^\ast\times\C^\ast$ is precisely the line $\{L=1\}$.  This shows that $L-1$ always divides $A_K(M,L)$.

In 2005, Dunfield and Garoufalidis \cite{dunfield-garoufalidis} and Boyer and Zhang \cite{boyer-zhang}  used Kronheimer and Mrowka's work in \cite{km-su2} to prove that $A_K(M,L)$ detects the unknot:

\begin{theorem} \label{thm:dg-bz-unknot}
 $A_K(M,L) = L-1$ if and only if $K$ is the unknot.
\end{theorem}

Our Theorem \ref{thm:thin-torus} vastly generalizes this result. Let $\tilde{V}(K)$ denote the plane curve defined in the same way as $V(K)$, except that we only consider pairs $(M,L)$ coming from \emph{irreducible} representations. Following Ni and Zhang \cite{ni-zhang}, we define an enhancement
\[ \tilde{A}_K(M,L)  \in \Z[M,L] \]
to be the defining polynomial of $\tilde{V}(K)$. As mentioned in the introduction, this \emph{enhanced} A-polynomial is related to the original by
\[A_K(M,L) = \operatorname{lcm}(\tilde{A}_K(M,L), L-1). \]
Indeed, $\tilde{A}_K(M,L)$ no longer \emph{has} to be a multiple of $L-1$, but this can still happen if there is a family of irreducible representations $\rho$ on which the trace of $\rho(\mu)$ takes infinitely many values while $\rho(\lambda)$ is always the identity matrix. An immediate consequence of Theorem \ref{thm:dg-bz-unknot} is that $\tilde{A}_K(M,L)$ detects the unknot as follows:

\begin{corollary} \label{cor:Atildeunknot}
 $\tilde{A}_K(M,L) = 1$ if and only if $K$ is the unknot.
\end{corollary}

Given a knot $K \subset S^3$, let \[\cN(K) \subset \R^2\] denote the \emph{Newton polygon} of $\tilde{A}_K(M,L)$, which is by definition the convex hull of all points $(a,b) \in \Z^2$ such that the $M^bL^a$-coefficient of $\tilde{A}_K(M,L)$ is nonzero. 
\begin{remark}
The fact that $(a,b)$ corresponds to a monomial $M^bL^a$, rather than to $M^aL^b$, is not a typo: this allows us to directly relate the edge slopes of $\cN(K)$ to slopes in the knot complement, rather than to their reciprocals. We will return to this point in \S\ref{sec:thin}.
\end{remark}

Corollary \ref{cor:Atildeunknot} is equivalent to the statement that $\cN(K)$ is a point if and only if $K$ is the unknot. The next simplest convex polygon after a point is a line segment, which leads to:

\begin{definition}
Let $r\in\Q$.  We say that a  knot $K \subset S^3$ is \emph{$r$-thin} if $\cN(K)$ is contained in a line segment of slope $r$. We say that $K$ is \emph{thin} if it is $r$-thin for some $r\in \Q$.
\end{definition}

The enhanced A-polynomial is very difficult to compute in general, but this computation is fairly easy for torus knots. The results of these computations are described in the example below, which shows in particular that a nontrivial torus knot $T_{a,b}$ is $ab$-thin.

\begin{example}\label{ex:torus}
Let $|a| > b \geq 2$. Then the enhanced A-polynomial of the torus knot $T_{a,b}$
is given by \[\tilde{A}_{T_{a,b}}(M,L) = \begin{cases}
1+M^{2a}L, &  b=2, \,a>0\\ 
M^{-2a}+L,&  b=2, \,a<0\\
-1+M^{2ab}L^2, &b>2, \,a>0\\ 
-M^{-2ab}+L^2, & b>2, \,a<0\\ 
\end{cases}\]
as stated in \cite[Lemma 2.1]{ni-zhang}, for instance. 
\end{example}

Our Theorem \ref{thm:thin-torus} asserts that  torus knots are the \emph{only} thin knots. Then \[\tilde{A}_K(M,L) = \tilde{A}_{T_{ab}}(M,L) \] implies that  $K = T_{p,q}$ where $pq=ab$. Corollary \ref{cor:detectionA} will follow from this together with the corollary below, which is an immediate consequence of the calculations in Example \ref{ex:torus}:

\begin{corollary}
\label{cor:detectionB}
$\tilde{A}_K(M,L)$ detects  $T_{a,b}$ among nontrivial torus knots if and only if:
\begin{itemize}
\item $|a|=2$ or $|b|=2$, or 
\item $|a|$ and $|b|$ are powers of different primes.
\end{itemize} 
\end{corollary}

The next result, which we proved in \cite[\S10]{bs-lspace}, is important because (1) it says that in thinking about thin knots we do not have to consider hyperbolic knots, and (2) it connects thinness with the notions of $\SU(2)$-averse knots and instanton L-spaces. This will allow us to study thin knots using instanton Floer techniques and the results in previous sections.

\begin{proposition} \label{prop:thin-basics}
If $K\subset S^3$ is a nontrivial $r$-thin knot then:
\begin{itemize}
\item $K$ is not hyperbolic, 
\item  $K$ is $\SU(2)$-averse with $r(K) = r$.
\end{itemize}
\end{proposition}

\begin{proof}
The first claim is \cite[Proposition~10.7]{bs-lspace}, which is really just a reinterpretation of \cite[Proposition~2.6]{ccgls}.  The second claim is  \cite[Proposition~10.10]{bs-lspace}.\end{proof}

Proposition~\ref{prop:thin-basics} says in particular that  any thin knot  is either a torus knot or a satellite knot. The torus knot case is completely understood, as described above. Our focus for the remainder of this paper will therefore be on thin satellite knots, and we record the following:

\begin{proposition} \label{prop:thin-satellite}
Suppose  $K = P(C)$ is an $r$-thin satellite knot and let $w$ be the winding number of the pattern $P$.  Then:
\begin{itemize}
\item $P(U)$ is nontrivial and $w \geq 2$,
\item $P(U)$ is $r$-thin,
\item the companion $C$ is $(r/w^2)$-thin.
\end{itemize}
\end{proposition}

\begin{proof}
Proposition~\ref{prop:thin-basics} says that $K$ is $\SU(2)$-averse. Then Corollary \ref{cor:SUL} says that either $K$ or $\mirror{K}$ is an instanton L-space knot, in which case  Lemma \ref{lem:satellite-lspace} implies that $P(U)$ is nontrivial and that $w \geq 1$. The second and third claims are then part of \cite[Lemma~10.12]{bs-lspace}.
It remains to prove that $w \geq 2$, for which it  suffices to show that $w\neq 1$.  

Since $P(U)$ and $C$ are $r$- and $(r/w^2)$-thin, they are $\SU(2)$-averse and thus up to mirroring are nontrivial instanton L-space knots, again by Proposition~\ref{prop:thin-basics} and Corollary \ref{cor:SUL}. Suppose their genera are $g \geq 1$ and $h\geq 1$, respectively.  Note that \[ \Delta_K(t) = \Delta_{P(U)}(t) \cdot \Delta_C(t^w). \] If $w=1$, then Corollary~\ref{cor:lspace-top-two} says that the right side of this equation
has the form
\[ \pm\big(t^g - t^{g-1} + O(t^{g-2})\big)\big(t^h - t^{h-1} + O(t^{h-2})\big) = \pm\big( t^{g+h} - 2t^{g+h-1} + O(t^{g+h-2}) \big). \]
But then the $t^{g+h-1}$-coefficient of $\Delta_K(t)$ would be $\pm2$, contradicting Theorem \ref{thm:ly-lspace}.  Therefore $w \neq 1$, completing the proof.
\end{proof}

\section{L-space satellite knots with torus knot pattern} \label{sec:satellites}

The results of the previous section show that understanding thin knots requires us to understand instanton L-space knots that are satellite knots.  In this section, we prove new  results about such knots; these will be key in our proofs of both Theorems \ref{thm:SLabeliantorus} and \ref{thm:thin-torus}. The proposition below is the main technical result of this section:

\begin{proposition} \label{prop:torus-satellite-non-lspace}
Let $K = P(C)$ be a satellite knot whose pattern $P$ has winding number $w \geq 1$.  If $P(U) = T_{a,b}$ is a nontrivial torus knot such that $ab$ is a multiple of $w^2$, and $C$ or $\mirror{C}$ is an instanton L-space knot, then neither $K$ nor $\mirror{K}$ is an instanton L-space knot.
\end{proposition}

We note that the hypothesis that $P(U)$ is a torus knot does not necessarily mean that $K = P(C)$ is a cable, because the pattern knot $P \subset S^1\times D^2$ does not have to be isotopic into the boundary torus $S^1 \times S^1$.

We will prove Proposition~\ref{prop:torus-satellite-non-lspace} by examining the Alexander polynomial of $K$, recalling that by convention this satisfies $\Delta_K(t) = \Delta_K(t^{-1})$ and $\Delta_K(1) = 1$.  We will show that $\Delta_K(t)$ cannot have the form guaranteed by Theorem~\ref{thm:ly-lspace}: either some coefficient has absolute value at least 2, or two consecutive nonzero coefficients have the same sign.  We begin by studying the leading terms of the Alexander polynomials of torus knots:

\begin{lemma} \label{lem:torus-poly-leading}
Let $g = g(T_{p,q}) = \frac{1}{2}(p-1)(q-1)$, where $p > q \geq 2$ and $\gcd(p,q)=1$.  Then
\[ \Delta_{T_{p,q}}(t) = \sum_{i=0}^{\lfloor p/q\rfloor} (t^{g-iq} - t^{g-iq-1}) + O(t^{g-p}). \]
\end{lemma}

\begin{proof}
Noting that $2g-1 = pq-p-q$, we write 
\begin{align*}
(t^p-1) t^g \Delta_{T_{p,q}}(t) &= \frac{(t^{pq}-1)(t-1)}{t^q-1} \\
&= (t-1)\left( t^{(p-1)q} + t^{(p-2)q} + \dots + t^q + 1 \right) \\
&= (t-1)\left( t^{2g+p-1} + t^{2g+p-1-q} + \dots + t^{2g+p-1 - q(p-1)}\right)
\end{align*}
and so
\begin{equation} \label{eq:alex-times-tp-1}
(t^p-1)\Delta_{T_{p,q}}(t) = (t-1)\left( t^{g+p-1} + t^{g+p-1-q} + \dots + t^{g+p-1 - q(p-1)}\right).
\end{equation}
Similarly, if we let $r=\lfloor \frac{p}{q} \rfloor$, then we have $q(r+1) > p$ and $1 \leq r \leq p-1$.  (The bound $r\leq p-1$ is immediate from $q \geq 2$, which implies that $r < p$.) We compute that
\begin{align*}
(t^p-1)\sum_{i=0}^{r} (t^{g-iq} - t^{g-iq-1}) &= (t^p-1)(t-1)\sum_{i=0}^r t^{g-1-iq} \\
&= (t-1)\sum_{i=0}^r t^{g+p-1-iq} - \underbrace{(t-1)\sum_{i=0}^r t^{g-1-iq}}_{=\,O(t^g)}.
\end{align*}
In the first term on the right, since $r \leq p-1$ we can extend the summation from $i=r$ to $i=p-1$ by adding additional terms, all of which are no bigger than
\[ (t-1) \cdot O(t^{g+p-1-q(r+1)}) = O(t^g). \]
Similarly the entire second term on the right is $O(t^g)$, so overall we have
\begin{equation} \label{eq:alex-approx-times-tp-1}
(t^p-1)\sum_{i=0}^{r} (t^{g-iq} - t^{g-iq-1}) = \sum_{i=0}^{p-1} t^{q+p-1-iq} + O(t^g).
\end{equation}
Comparing \eqref{eq:alex-times-tp-1} and \eqref{eq:alex-approx-times-tp-1} gives
\[ (t^p-1)\left(\Delta_{T_{p,q}}(t) - \sum_{i=0}^{r} (t^{g-iq} - t^{g-iq-1})\right) = O(t^g), \]
and so by comparing degrees we conclude that
\[ \Delta_{T_{p,q}}(t) - \sum_{i=0}^{r} (t^{g-iq} - t^{g-iq-1}) = O(t^{g-p}). \qedhere \]
\end{proof}

\begin{lemma} \label{lem:w-mod-b}
Suppose that $K=P(C)$ is a satellite knot, where
\begin{itemize}
\item $P(U)=T_{a,b}$ for some relatively prime integers $a$ and $b$, with $a > b \geq 2$;
\item and either $C$ or $\mirror{C}$ is an instanton L-space knot.
\end{itemize}
If either $K$ or $\mirror{K}$ is an instanton L-space knot, and if $P$ has winding number $w < a$, then $w$ is a positive multiple of $b$.
\end{lemma}

\begin{proof}
We note that $w \geq 1$ by Lemma~\ref{lem:satellite-lspace}, so $1 \leq w \leq a-1$.  Now Lemma~\ref{lem:torus-poly-leading} says that if $g=g(T_{a,b})$ then
\[ \Delta_{T_{a,b}}(t) = \sum_{i=0}^{\lfloor a/b\rfloor} (t^{g-ib} - t^{g-ib-1}) + O(t^{g-a}). \]
Certainly $t^{g-a} = O(t^{g-w-1})$ since $a \geq w+1$, so we can read the $t^{g-w}$-coefficient of $\Delta_{T_{a,b}}(t)$ directly off of this sum: it is
\[ c_{g-w} = \begin{cases} 
1, & w\equiv0 \pmod{b} \\
-1, & w\equiv1 \pmod{b} \\
0, & \text{otherwise}.
\end{cases} \]

We let $h=g(C)$, and noting that $\pm \Delta_C(t) = t^h - t^{h-1} + O(t^{h-2})$ by Corollary~\ref{cor:lspace-top-two}, we examine the $t^{g+hw-w}$-coefficient of 
\begin{align*}
\pm\Delta_K(t) &= \pm \Delta_{T_{a,b}}(t) \cdot \Delta_{C}(t^w) \\
&= \left(t^g - t^{g-1} + \dots + c_{g-w}t^{g-w} + O(t^{g-w-1}) \right) \left(t^{hw} - t^{(h-1)w} + O(t^{(h-2)w})\right).
\end{align*}
By inspection, this coefficient is equal to $c_{g-w} - 1$, and if $c_{g-w}=-1$ then $\Delta_K(t)$ has a $t^{g+hw-w}$-coefficient of $-2$, contradicting Theorem~\ref{thm:ly-lspace}.  (It does not matter here whether $K$ or its mirror $\mirror{K}$ is an instanton L-space knot, since $\Delta_K(t) = \Delta_{\mirror{K}}(t)$.)  Thus $w \not\equiv 1\pmod{b}$.

Now if $w\not\equiv 0 \pmod{b}$, then the only remaining possibility is that $c_{g-w}=0$ and so this $t^{g+hw-w}$-coefficient is equal to $-1$.  But in this case, if we write
\[ w = qb + r, \qquad 2 \leq r < b \]
then Lemma~\ref{lem:torus-poly-leading} (plus the fact that $g-a \leq g-w-1$)  tells us that
\[ \Delta_{T_{a,b}}(t) = (t^g - t^{g-1}) + (t^{g-b}-t^{g-b-1}) + \dots + (t^{g-qb}-t^{g-qb-1}) + O(t^{g-w-1}), \]
and  we compute that
\begin{align*}
\pm\Delta_K(t) &= \pm\Delta_{T_{a,b}}(t) \cdot \Delta_{C}(t^w) \\
&= \left(t^g - \dots + t^{g-qb} - t^{g-qb-1} + O(t^{g-w-1}) \right) \left(t^{hw} - t^{(h-1)w} + O(t^{(h-2)w})\right) \\
&= t^{g+hw} + \dots + t^{g+hw-qb} - t^{g+hw-qb-1} - t^{g+hw-w} + O(t^{g+hw-w-1}).
\end{align*}
In this case two consecutive nonzero coefficients of $\Delta_K(t)$, namely those corresponding to $t^{g+hw-qb-1}$ and $t^{g+hw-w}$, have the same sign.  This also contradicts Theorem~\ref{thm:ly-lspace}, so this case cannot occur either.  We conclude that $w$ must have been a multiple of $b$ after all.
\end{proof}

\begin{proof}[Proof of Proposition~\ref{prop:torus-satellite-non-lspace}]
Suppose that $K = P(C)$ is a satellite knot with $P(U) = T_{a,b}$ for some $a,b$, with $C$ or $\mirror{C}$ an instanton L-space knot, and such that the winding number $w \geq 1$ of $P$ satisfies
\[ ab = kw^2 \]
for some integer $k$.  If $T_{a,b}$ is a negative torus knot then we can replace $K$ with $\mirror{K}$, and this replaces $T_{a,b}$ with $T_{-a,b}$ but does not change the Alexander polynomial $\Delta_K(t)$, so we may as well assume that $a > b \geq 2$ and then $k \geq 1$.

We must now have $w < a$: assuming otherwise, we would have
\[ ab = kw^2 \geq a^2 > ab, \]
which is absurd since $ab > 0$.  Thus if either $K$ or $\mirror{K}$ is an instanton L-space knot, then Lemma~\ref{lem:w-mod-b} says that $w$ is a multiple of $b$.  But now
\[ a = b \cdot k\left(\frac{w}{b}\right)^2 \]
tells us that $a$ is an integer multiple of $b$, and this is impossible since $a$ and $b$ must be relatively prime, so neither $K$ nor $\mirror{K}$ can be an instanton L-space knot after all.
\end{proof}

\section{Thin knots: the proof of Theorem \ref{thm:thin-torus}}
\label{sec:thin}

In this section, we prove Theorem \ref{thm:thin-torus}. Example \ref{ex:torus} showed that nontrivial torus knots are thin.  It remains to prove the other direction of Theorem \ref{thm:thin-torus}, stated  as follows:

\begin{theorem} \label{thm:thin-torus-proof}
If $K \subset S^3$ is a nontrivial $r$-thin knot, then $K=T_{a,b}$ where $r=ab$.
\end{theorem}

\begin{proof}
Note that $K$ is not hyperbolic, by Proposition \ref{prop:thin-basics}. If $K$ is a torus knot then we are done, so it suffices to prove that there are no thin satellite knots. 

Suppose for a contradiction that $K=P(C)$ is an $r$-thin satellite knot with the smallest genus amongst all thin satellite knots.  Proposition \ref{prop:thin-satellite} says that $P(U)$ and $C$ are  nontrivial $r$- and $(r/w^2)$-thin knots, respectively, and that $w\geq 2$. Since $P(U)$ and $C$ are thin, Proposition~\ref{prop:thin-basics} tells us that they are not hyperbolic.  Then from the relation
\[ g(K) = g(P(U)) + w\cdot g(C) \]
(see for example \cite[Proposition~2.10]{burde-zieschang}), we see that $P(U)$ and $C$ have  genera strictly less than that of $K$, and must therefore be torus knots. If we write $P(U) = T_{a,b}$ and $C = T_{c,d}$ then \[ab =r \textrm{ and } cd = r/w^2.\] In particular, $ab = cdw^2$. Since torus knots are instanton L-space knots, Proposition \ref{prop:torus-satellite-non-lspace} then implies that neither $K$ nor $\mirror{K}$ is an instanton L-space knot. But this yields a contradiction: since $K$ is thin, it is $\SU(2)$-averse by Proposition \ref{prop:thin-basics}, which implies by Corollary \ref{cor:SUL} that either $K$ or its mirror must be an instanton L-space knot. 
\end{proof}

\begin{proof}[Proof of Corollary \ref{cor:detectionA}]
This follows immediately from Theorem \ref{thm:thin-torus-proof} and Corollary \ref{cor:detectionB}.
\end{proof}

\begin{proof}[Proof of Corollary \ref{cor:detectionAdeg}]
Suppose $K\subset S^3$ is a knot for which the pair  \[\big(\tilde{A}_K(M,L), \deg \Delta_K(t)\big)\] agrees with the corresponding pair for the torus knot $T_{a,b}$. If $K$ is the unknot then the result holds by Corollary \ref{cor:Atildeunknot}, so we may assume that $K$ is nontrivial. Then $K$ is a torus knot $T_{p,q}$ with $pq=ab$ by Theorem \ref{thm:thin-torus-proof}, by comparing the enhanced A-polynomials. Comparing the degrees of the Alexander polynomials, we also have that $(p-1)(q-1) = (a-1)(b-1).$ These two equations imply that $\{p,q\}=\{a,b\}$ and hence that $K = T_{a,b}$.
\end{proof}

\subsection{Boundary slopes} In this section, we use Theorem \ref{thm:thin-torus} to prove Theorem \ref{thm:slopes}. We first provide some background on boundary slopes, elaborating on the discussion in \S\ref{sec:intro}.

Given a knot $K\subset Y$, let $E_K = Y\setminus \nu(K)$ denote the exterior of a tubular neighborhood of $K$. A peripheral slope $\gamma \subset \partial E_K$ is called a \emph{boundary slope} for $K$ if there is a properly embedded, incompressible, boundary-incompressible, orientable surface $F \subset E_K$ whose boundary $\partial F \subset \partial E_K$ is a union of one or more curves of slope $\gamma$.  If $F$ is not a fiber in any fibration $E_K \to S^1$, then we say that $\gamma$ is a \emph{strict} boundary slope.

\begin{example} \label{ex:seifert-surface}
Every knot $K \subset S^3$ has boundary slope $0$, realized as the boundary of a Seifert surface of genus $g(K)$.
\end{example}

\begin{example}
A nontrivial torus knot $T_{a,b}$ has boundary slope $ab$: if we view $T_{a,b}$ as a curve on a standardly embedded torus $\mathbb{T} \subset S^3$, then $\mathbb{T} \setminus N(T_{a,b})$ is an incompressible annulus of boundary slope $ab$.  In fact, $0$ and $ab$ are the only boundary slopes for $T_{a,b}$.
\end{example}

If the knot exterior $E_K$ is orientable, compact, and irreducible, then $K$ has only finitely many boundary slopes \cite{hatcher-slopes}.  Nontrivial knots in $S^3$ have at least two boundary slopes \cite{cs-surfaces}, and conjecturally only torus knots have exactly two, as discussed in the introduction.  The A-polynomial is an effective source of boundary slopes in the following sense, as proven in {\cite[Theorem~3.4]{ccgls}}:

\begin{theorem}[] \label{thm:edge-slope}
If $r\in\Q\cup\{\infty\}$ is the slope of some edge of the Newton polygon $\cN(K)$ of $\tilde{A}_K(M,L)$, then $r$ is a strict boundary slope for $K$.
\end{theorem}

\begin{remark}
The theorem in \cite{ccgls} uses $A_K(M,L)$ rather than $\tilde{A}_K(M,L)$, and does not claim that $r$ should be strict; indeed, we know that $0$ is always an edge slope for the Newton polygon of $A_K(M,L)$, but it might not be strict as a boundary slope if $K$ is fibered.  But if $0$ is an edge slope of $\tilde{A}_K(M,L)$, then \cite[\S3]{ccgls} associates to this edge an ideal point of the character variety of $E_K$, and following \cite{cs-splittings} this gives rise to an action of $\pi_1(E_K)$ on a tree $T$, from which we extract the incompressible surface $S$ of boundary slope $0$.  By construction $\pi_1(S)$ stabilizes some edge of $T$, so by \cite[Proposition~1.2.7]{cgls} it can't be a normal subgroup of $\pi_1(E_K)$, hence $S$ can't be a fiber in a fibration $E_K \to S^1$.
\end{remark}

As mentioned  above, $0$ is always an edge slope of the Newton polygon of $A_K(M,L)$, corresponding to its $L-1$ factor. Its nonzero edge slopes are then precisely those of the Newton polygon $\cN(K)$ of $\tilde{A}_K(M,L)$.  Thus we have the following:

\begin{lemma} \label{lem:0-strict}
Let $K \subset S^3$ be a nontrivial knot with exactly two boundary slopes.  If $K$ is not a torus knot, then $0$ is an edge slope of $\cN(K)$ and thus a strict boundary slope for $K$. 
\end{lemma}

\begin{proof}
Since $K$ is not a torus knot, Theorem~\ref{thm:thin-torus} tells us that it is not thin: that is, the Newton polygon $\cN(K)$ is not contained in a line.  Then $\cN(K)$ has at least two distinct edge slopes, and Theorem~\ref{thm:edge-slope} says that each of these is a strict boundary slope for $K$.  But $K$ only has two boundary slopes in the first place, so in fact every boundary slope for $K$ is both an edge slope for $\cN(K)$ and strict.  Example~\ref{ex:seifert-surface} now says that $0$ is a boundary slope for $K$, so it must be a strict boundary slope after all.
\end{proof}

We will also need the following fact, which appears to be classically known but we could not find a reference:

\begin{proposition} \label{prop:nonseparating-surface-in-fibered-exterior}
Let $K \subset S^3$ be a fibered knot with exterior $E_K$.  Let $F \subset E_K$ be a fiber surface, and let $S \subset E_K$ be a connected, orientable, nonseparating, incompressible surface whose boundary meets $\partial E_K$ in a union of longitudes.  Then $S$ is isotopic to $F$.
\end{proposition}

\begin{proof}
Since $S$ is nonseparating, there is an arc in $E_K \setminus N(S)$ from one side of $S$ to the other; we join the ends of this arc inside $N(S)$ to get a closed curve $c \subset E_K$ that meets $S$ transversely in one point.  The intersection pairing $\langle [S], [c] \rangle$ is then equal to $\pm1$, so $[S]$ must be primitive as an element of $H_2(E_K,\partial E_K) \cong \Z$.  In other words, we have $[S] = \pm [F]$, and so the boundary homomorphism
\[ \partial: H_2(E_K, \partial E_K) \to H_1(\partial E_K) \]
sends $[S]$ to $\pm [\partial F] = \pm [\lambda_K]$.  In other words, $[\partial S] = \pm[\partial F]$ in $H_1(\partial E_K)$.

We now let $p: \tilde{E} \to E_K$ denote the $\Z$-fold cyclic cover corresponding to the kernel of the abelianization map $\ab: \pi_1(E_K) \twoheadrightarrow \Z$, and observe that $\tilde{E} \cong F \times \R$.  Given any closed curve $\gamma \subset S$, we can push $\gamma$ off of $S$ inside $E_K$ to conclude that
\[ \operatorname{lk}(\gamma,K) = \langle [F], [\gamma] \rangle = \pm \langle [S], [\gamma] \rangle = 0. \]
Thus $\gamma \in \ker(\ab)$, and since $\gamma$ was arbitrary we have $\pi_1(S) \subset \ker(\ab) = p_*(\pi_1(\tilde{E}))$.  So $S$ lifts to a surface $\tilde{S} \subset \tilde{E}$, and moreover any compressing disk for $\tilde{S}$ could be pushed down to a compressing disk for $S$ itself, so $\tilde{S}$ must be incompressible.

We now let $i: \tilde{S} \hookrightarrow \tilde{E}$ denote inclusion, and consider the composition
\[ (\phi,h): \tilde{S} \xrightarrow{i} \tilde{E} \cong F \times \R. \]
We note that $i_*: \pi_1(\tilde{S}) \to \pi_1(\tilde{E})$ is injective since $\tilde{S}$ is incompressible, and the projection $p_F: F\times\R \to F$ induces an isomorphism of fundamental groups, so the map
\[ \phi_* = (p_F \circ (\phi,h))_*: \pi_1(\tilde{S}) \to \pi_1(F) \]
is injective.  Since $\phi$ is a map between compact surfaces inducing an injection of fundamental groups, a theorem of Nielsen \cite{nielsen-surfaces} says that it is homotopic to a covering map $\psi: \tilde{S} \to F$.  We use $\psi$ to pull back the orientation of $F$ to $\tilde{S}$, and then the projection $p_F$ is an orientation-preserving homeomorphism from each component of $\partial \tilde{S} \subset \partial F \times \R$ to $\partial F$.  It follows that
\[ [\partial \tilde{S}] = (\deg \psi)[\partial F] \in H_1(\partial F \times \R), \]
since every component of $\partial \tilde{S}$ is counted with the same sign.  Projecting from $\tilde{E} \cong F\times \R$ to the knot exterior $E_K$, we conclude that
\[ [\partial S] = (\deg\psi) [\partial F] \in H_1(E_K); \]
but we showed above that $[\partial S] = \pm[\partial F]$, so we must have $\deg\psi = 1$.  In other words, the covering map $\psi: \tilde{S} \to F$ is a homeomorphism, and $\tilde{S}$ has connected boundary, which up to translation we can take to be $\partial F \times \{0\}$.

Now the incompressible surfaces $\tilde{S}$ and $F\times \{0\}$ in $\tilde{E} \cong F\times\R$ are homotopic rel boundary, via a homotopy $(\phi,h)\simeq (\psi,0)$.  Composing with the projection $p:\tilde{E} \to E_K$, we see that $S$ is homotopic rel boundary to $F$.  A result of Waldhausen \cite[Corollary~5.5]{waldhausen} now tells us that $S$ is isotopic rel boundary to $F$, as desired.
\end{proof}

We are now ready to prove Theorem~\ref{thm:slopes} from the introduction, restated as follows:

\begin{theorem} \label{thm:two-slopes}
Let $K \subset S^3$ be a nontrivial fibered knot with exactly two boundary slopes.  If $K$ is not a torus knot, then its exterior $E_K$ contains a connected, separating, incompressible surface $S$ which meets $\partial E_K$ in curves of slope $0$.
\end{theorem}

\begin{proof}
Let $F \subset E_K$ be a fiber of the fibration $E_K \to S^1$.  Lemma~\ref{lem:0-strict} says that $0$ is a strict boundary slope for $K$; if $S \subset \partial E_K$ is the corresponding connected, incompressible surface, then $S$ is not isotopic to $F$.  Now Proposition~\ref{prop:nonseparating-surface-in-fibered-exterior} says that if $S$ were nonseparating then it would have to be isotopic to $F$, so $S$ must separate $E_K$ after all.
\end{proof}

\subsection{$\PSL(2,\C)$-averse knots}

Proving Theorem \ref{thm:SLabeliantorus}, which states that all $\SL(2,\C)$-averse knots are torus knots, will take a good deal more work, so we postpone it to \S\ref{sec:SL}.  Here we show instead that the analogue of Theorem \ref{thm:SLabeliantorus} with $\PSL(2,\C)$ in place of $\SL(2,\C)$ follows readily from Theorem \ref{thm:thin-torus}:

\begin{theorem} \label{thm:psl2-averse}
If $K \subset S^3$ is a $\PSL(2,\C)$-averse knot, then $K$ is a torus knot.
\end{theorem}

Suppose that $K\subset S^3$ is a knot and let us  recall some facts about the $\SL(2,\C)$ character variety $X(E_K)$ associated to the knot exterior $E_K = S^3\setminus\nu(K)$.  These facts are taken from \cite[\S5]{boyer-zhang-seminorms}, where they generalize results of \cite[\S1]{cgls}.

Given a peripheral slope $\gamma \subset \partial E_K$, the trace function
\begin{align*}
f_\gamma: X(E_K) &\to \C \\
[\rho] &\mapsto \tr(\rho(\gamma))^2 - 4
\end{align*}
is a regular function on $X(E_K)$.  An irreducible curve $X_0 \subset X(E_K)$ defines a \emph{Culler--Shalen seminorm} $||\cdot||_{X_0}$ on $H_1(\partial E_K;\R)$, satisfying exactly one of the following:
\begin{itemize}
\item $||\cdot||_{X_0}$ is a norm, and every function $f_\gamma$ is non-constant on $X_0$.
\item $||\cdot||_{X_0}$ is a indefinite seminorm but is not identically zero, and exactly one function $f_{\gamma_0}$ is constant on $X_0$.
\item $||\cdot||_{X_0}$ is zero, and every $f_\gamma$ is constant on $X_0$.
\end{itemize}
In the first two cases, we call $X_0$ a \emph{norm curve} or a \emph{seminorm curve}, respectively.  The following two propositions about such curves are stated in the form seen here in \cite{ni-zhang}, but essentially appear in \cite[\S6]{boyer-zhang-seminorms}.

First, if $X_0$ is a norm curve, then we let $s_0 > 0$ be the minimal positive value of $||\alpha||_{X_0}$ over all integral classes $\alpha \in H_1(\partial E_K;\Z)$, and then we let $B_0 \subset H_1(\partial E_K;\R)$ be the convex, finite-sided polygon on which $||\cdot||_{X_0} \leq s_0$. Ni and Zhang state  the following in {\cite[Theorem~2.2(4)]{ni-zhang}}: 

\begin{proposition}[] \label{prop:psl-norm}
If $X_0 \subset X(E_K)$ is a norm curve and $S^3_r(K)$ is $\PSL(2,\C)$-abelian for some $r=p/q$, then either $r$ is a boundary slope for $K$, or $(p,q)$ lies on the boundary $\partial B_0$ but is not a vertex of $B_0$.
\end{proposition}

If instead $X_0$ is a seminorm curve, and the unique $\gamma$ for which $f_\gamma|_{X_0}$ is constant has slope $r$, then we say that $X_0$ is an \emph{$r$-curve}.  For example, the curve of abelian characters of $\pi_1(E_K)$ is a $0$-curve, since $f_\lambda \equiv 0$ on that curve. Ni and Zhang state the following regarding such curves  {\cite[Theorem~2.5(3)]{ni-zhang}}:

\begin{proposition}[] \label{prop:psl2-seminorm}
If $X_0 \subset X(E_K)$ is an $r$-curve containing the character of an irreducible representation, and if $S^3_s(K)$ is $\PSL(2,\C)$-abelian for some slope $s$, then either $s$ is a boundary slope for $K$, or $\Delta(r,s) = 1$.
\end{proposition}

Combining these facts with Theorem~\ref{thm:thin-torus-proof} allows us to prove Theorem \ref{thm:psl2-averse}:

\begin{proof}[Proof of Theorem \ref{thm:psl2-averse}]
Suppose that the character variety $X(E_K)$ contains a norm curve.  Then $\partial B_0$ contains only finitely many integral classes, and a theorem of Hatcher \cite{hatcher-slopes} says that $K$ has only finitely many boundary slopes, so according to Proposition~\ref{prop:psl-norm} there can only be finitely many $\PSL(2,\C)$-abelian slopes for $K$.

Now suppose instead that there are $r\neq s$ such that $X(E_K)$ contains an $r$-curve and an $s$-curve, both of which contain an irreducible character.  Then Proposition~\ref{prop:psl2-seminorm} says that if $t$ is a $\PSL(2,\C)$-abelian slope for $K$, then either $t$ is a boundary slope, or $\Delta(r,t) = \Delta(s,t) = 1$.  There are only finitely many of the former, and at most two of the latter, so again $K$ has finitely many $\PSL(2,\C)$-abelian slopes.

In the remaining cases, there is some fixed $r=p/q$ such that every curve $X_0 \subset X(E_K)$ containing an irreducible character must satisfy either
\begin{itemize}
\item $X_0$ is an $r$-curve, or
\item every function $f_\gamma$ is constant on $X_0$.
\end{itemize}
Let $\gamma_0 \subset \partial E_K$ denote the slope corresponding to $r$.  We claim that $f_{\gamma_0}$ must be constant on any irreducible component $C_0 \subset X(E_K)$ that contains the character of an irreducible representation.  Suppose otherwise, for a contradiction. Then  $f_{\gamma_0}$ is not even  locally constant since $C_0$ is connected. It follows that at any smooth point $\chi \in C_0$ corresponding to an irreducible representation, we can intersect $C_0$ with a suitably chosen affine subspace through $\chi$ to get a curve $X_0$ on which $f_{\gamma_0}$ is still not locally constant, and some irreducible component of $X_0$ will not be an $r$-curve, a contradiction.

The character variety $X(E_K)$ only has finitely many irreducible components, so it follows that $f_{\gamma_0}([\rho])$ and hence $\tr(\rho(\mu^p\lambda^q))$ take only finitely many values as $\rho$ ranges over all irreducible representations of $\pi_1(E_K)$.  But this means that the plane curve used to define $\tilde{A}_K(M,L)$ belongs to the zero set of  $f(M^pL^q)$ for some polynomial $f\in \Z[t]$, so $\tilde{A}_K(M,L)$ is itself a polynomial in $M^pL^q$.  This means that $K$ is $p/q$-thin, so Theorem~\ref{thm:thin-torus-proof} now guarantees that $K$ is a torus knot.
\end{proof}

\section{$\SL(2,\C)$-averse knots: the proof of Theorem \ref{thm:SLabeliantorus}}
\label{sec:SL}

In this section, we prove Theorem \ref{thm:SLabeliantorus}. Note that nontrivial torus knots are $\SL(2,\C)$-averse, again because  $(ab+1/n)$-surgery on $T_{a,b}$ is a lens space for every $n\in \Z$ and hence $\SL(2,\C)$-abelian.  It thus  remains to prove the other direction of  Theorem \ref{thm:SLabeliantorus}, stated as:

\begin{theorem} \label{thm:SLabeliantorus-proof}
If $K \subset S^3$ is an $\SL(2,\C)$-averse knot, then $K$ is a torus knot.
\end{theorem}

We first observe the following:

\begin{lemma} \label{lem:sl2-averse-hyperbolic}
Hyperbolic knots are never $\SL(2,\C)$-averse.
\end{lemma}

\begin{proof}
If $K$ is hyperbolic, then Thurston's hyperbolic Dehn surgery theorem \cite[Theorem~5.8.2]{thurston-notes} says that all but finitely many Dehn surgeries on $K$ are closed hyperbolic manifolds.  If $S^3_r(K)$ is hyperbolic, then there is a faithful representation \[\pi_1(S^3_r(K)) \to \SL(2,\C)\] \cite[Proposition~3.1.1]{cs-splittings}, which has nonabelian image.  Therefore, only the finitely many non-hyperbolic surgeries can possibly be $\SL(2,\C)$-abelian.
\end{proof}

Therefore, to prove Theorem \ref{thm:SLabeliantorus}, we just need to show that there are no $\SL(2,\C)$-averse satellite knots. This is the aim of the rest of the section, and we proceed in a manner that is similar in spirit to our proof that there are no thin satellite knots. 

\begin{lemma} \label{lem:sl2-abelian-pinch}
Let $r\in\Q$.  If $r$-surgery on a satellite knot $K = P(C)$ is $\SL(2,\C)$-abelian, then so is $r$-surgery on $P(U)$.
\end{lemma}

\begin{proof}
Our proof is essentially the same as Sivek and Zentner's proof of \cite[Proposition~10.2]{sz-pillowcase}, which is the analogue of this lemma with $\SU(2)$ in place of $\SL(2,\C)$.  In particular, we use the standard fact that there is a degree-1 map pinching the exterior $E_C = S^3 \setminus \nu(C)$ onto the unknot exterior $E_U = S^1\times D^2$, that restricts to a homeomorphism
\begin{align*}
\partial E_C &\to \partial E_U \\
\mu_C,\lambda_C &\mapsto \mu_U,\lambda_U.
\end{align*}
We can glue the pattern exterior to each of these, and extend this pinching map across it as the identity map, to get a degree-1 map
\[ E_K = E_{P(C)} \to E_{P(U)} \]
sending meridians to meridians and longitudes to longitudes.  The induced map on $\pi_1$ is surjective and takes $\mu_K^p \lambda_K^q$ to $\mu_{P(U)}^p \lambda_{P(U)}^q$ where $r=p/q$, so it descends to a surjection
\[ \pi_1(S^3_r(K)) \to \pi_1(S^3_r(P(U))). \]
Any nonabelian representation
\[ \pi_1(S^3_r(P(U))) \to \SL(2,\C) \]
would therefore induce a  nonabelian representation of $\pi_1(S^3_r(K))$, and since the latter does not exist it follows that $r$-surgery on $P(U)$ must be $\SL(2,\C)$-abelian as well.
\end{proof}

\begin{lemma} \label{lem:sl2-abelian-companion}
Let $K = P(C)$ be a satellite knot whose pattern $P$ has winding number $w \geq 1$, and let $r\in\Q$.  If $r$-surgery on $C$ is not $\SL(2,\C)$-abelian, then neither is $rw^2$-surgery on $K$.
\end{lemma}

The analogue of this lemma with $\SU(2)$ in place of $\SL(2,\C)$ was first proved by Sivek and Zentner in \cite[Proposition~10.3]{sz-pillowcase}. Our proof follows the same ideas, but is considerably longer: the  proof in \cite{sz-pillowcase} begins by assuming that a given representation of $\pi_1(E_C)$ into $\SU(2)$ is diagonal on the peripheral subgroup (corresponding to Case 1 in our proof below), but here we must consider two additional cases in which the peripheral values are parabolic elements of $\SL(2,\C)$.

\begin{proof}[Proof of Lemma \ref{lem:sl2-abelian-companion}]
Write $r = p/q$, with $p$ and $q$ relatively prime and $q > 0$.  Let $d=\gcd(q,w^2)$, and note that
\[ rw^2 = \frac{pw^2/d}{q/d} \]
in lowest terms, since $\gcd(pw^2/d, q/d) = 1$.  Let $E_C = S^3 \setminus \nu(C)$ denote the exterior of $C$, and fix a nonabelian representation
\begin{align*}
\rho: \pi_1(E_C) &\to \SL(2,\C) \\
\mu_C^p \lambda_C^q &\mapsto 1,
\end{align*}
which exists by hypothesis.  We will assume, by replacing $\rho$ with a conjugate if necessary, that $\rho(\mu_C)$ is in Jordan normal form.  This means that $\rho(\mu_C)$ is a diagonal matrix unless its eigenvalues are equal, in which case these equal eigenvalues are  both $+1$ or both $-1$  since $\det(\rho(\mu_C)) = 1$.  We also note that $\rho(\mu_C)$ cannot be either of the central elements $\pm I$, since $\rho(\mu_C)$  normally generates the nonabelian image of $\rho$.

Letting $V = (S^1\times D^2) \setminus \nu(P)$, we may decompose the exterior $E_K = S^3\setminus \nu(K)$ as
\[ E_K \cong E_C \cup_T V. \]
Our goal will be to extend $\rho$ across $\pi_1(V)$, by constructing an \emph{abelian} representation
\[ \rho_V: \pi_1(V) \to \SL(2,\C) \]
such that
\begin{align*}
\rho_V(\mu_C) = \rho_V( [\{\pt\} \times \partial D^2] ) &= \rho(\mu_C) \\
\rho_V(\lambda_C) = \rho_V( [S^1 \times \{\pt\}] ) &= \rho(\lambda_C) \\
\rho_V(\mu_P^{pw^2/d} \lambda_P^{q/d}) &= 1.
\end{align*}
Then $\rho$ and $\rho_V$ will glue together to give a representation \[\pi_1(E_K) \to \SL(2,\C)\] that descends to a nonabelian representation of $\pi_1(S^3_{rw^2}(K))$, as desired.

Observe that any abelian representation $\rho_V: \pi_1(V) \to \SL(2,\C)$ must factor through 
\[ H_1(V;\Z) \cong \Z^2, \]
which is generated by the classes $[\mu_P]$ and $[\lambda_C]$ subject to the relations
\[
[\mu_C] = w [\mu_P]\, \textrm{ and }\,
[\lambda_P] = w[\lambda_C].
\]
The desired $\rho_V$ is therefore determined by the elements $\rho_V(\mu_P)$ and $\rho_V(\lambda_C)$, and the value $\rho_V(\lambda_C) = \rho(\lambda_C)$ is already fixed by $\rho$ -- we note that this implies that $\rho_V(\lambda_P) = \rho(\lambda_C)^w$ -- so we wish to find $\rho_V(\mu_P) \in \SL(2,\C)$ satisfying
\begin{align*}
\rho_V(\mu_P)^w &= \rho(\mu_C) \\
\rho_V(\mu_P^{pw^2/d} \lambda_P^{q/d}) &= 1.
\end{align*}
We consider several cases separately, recalling that $\rho(\mu_C)$ is in Jordan normal form.

\vspace{1em}\noindent
\underline{Case 1}: $\rho(\mu_C) \in \SL(2,\C)$ is a diagonal matrix $\left(\begin{smallmatrix} \alpha & 0 \\ 0 & \alpha^{-1} \end{smallmatrix}\right)$, where $\alpha \neq \pm1$.\\

In this case, $\rho(\mu_C)$ is diagonal but not central, and $\rho(\lambda_C)$ commutes with it, so $\rho(\lambda_C)$ must be diagonal as well: we can write
\begin{align*}
\rho(\mu_C) &= \begin{pmatrix} \alpha & 0 \\ 0 & \alpha^{-1} \end{pmatrix}, &
\rho(\lambda_C) &= \begin{pmatrix} \beta & 0 \\ 0 & \beta^{-1} \end{pmatrix}
\end{align*}
where $\alpha^p \beta^q = 1$.  We let $\alpha = s e^{i\theta}$ and $\beta = t e^{i\phi}$, where $s$ and $t$ are positive real numbers and $\theta,\phi\in\R$, and then the condition $\alpha^p\beta^q=1$ means that
\[
s^pt^q = 1 \,\textrm{ and }\,
p\theta + q\phi = 2\pi m
\]
for some integer $m$.  Let us write
\[ \eta = s^{1/w} \exp\left( \frac{(\theta+2\pi k)i}{w} \right) \]
for some integer $k$, chosen so that 
\[ m+pk \equiv 0 \pmod{d}; \]
we note that this is possible since $d=\gcd(q,w^2)$ divides $q$ and is thus coprime to $p$, meaning that $p$ is invertible $\!\!\!\pmod{d}$.   Then $\eta^w = \alpha$, so we set
\begin{align*}
\rho_V(\mu_P) &= \begin{pmatrix} \eta & 0 \\ 0 & \eta^{-1} \end{pmatrix}, &
\rho_V(\lambda_P) &= \begin{pmatrix} \beta^w & 0 \\ 0 & \beta^{-w} \end{pmatrix}
\end{align*}
so that $\rho_V(\mu_P)^w = \rho(\mu_C)$ and $\rho_V(\lambda_P) = \rho(\lambda_C)^w$.

Having done this, we now compute that
\[ \rho_V(\mu_P^{pw^2/d} \lambda_P^{q/d}) = \begin{pmatrix} z & 0 \\ 0 & z^{-1} \end{pmatrix}, \]
where $z = \eta^{pw^2/d} \cdot (\beta^w)^{q/d}$ satisfies
\begin{align*}
z &= s^{pw/d} t^{qw/d} \cdot \exp\left( \frac{pw(\theta+2\pi k)i}{d} + \frac{\phi i \cdot qw}{d}\right) \\
&= (s^pt^q)^{w/d} \cdot \exp\left( \frac{\big((p\theta+q\phi) + 2\pi pk\big)wi }{d} \right) \\
&= \exp\left( 2\pi \left(\frac{m+pk}{d}\right) \cdot wi \right)
\end{align*}
since $s^pt^q = 1$ and $p\theta + q\phi = 2\pi m$.  But we chose $k$ so that $(m+pk)/d$ would be an integer, so $z=1$ and thus $\rho_V(\mu_P^{pw^2/d}\lambda_P^{q/d}) = 1$, meaning that $\rho_V$ is the desired representation.

\vspace{1em}\noindent
\underline{Case 2}: $\rho(\mu_C)$ is a $2\times 2$ Jordan block with  eigenvalues which are either both $1$ or both $-1$, where  in the latter case we require that $w$ is odd.\\

Since $\rho(\mu_C) \neq \pm1$ commutes with $\rho(\lambda_C)$,  the latter is also a Jordan block. So, we have
\begin{align*}
\rho(\mu_C) &= \epsilon \begin{pmatrix} 1 & a \\ 0 & 1 \end{pmatrix}, &
\rho(\lambda_C) &= \eta \begin{pmatrix} 1 & b \\ 0 & 1 \end{pmatrix}
\end{align*}
for some $a,b\in\C$ with $a\neq 0$ and some $\epsilon,\eta=\pm1$ with $\epsilon^w = \epsilon$.  Then $\rho(\mu_C^p\lambda_C^q)=1$ is equivalent to $\epsilon^p\eta^q=1$ and $ap+bq=0$.  We define $\rho_V$ by
\begin{align*}
\rho_V(\mu_P) &= \epsilon \begin{pmatrix} 1 & a/w \\ 0 & 1 \end{pmatrix}, &
\rho_V(\lambda_P) &= \eta^w \begin{pmatrix} 1 & bw \\ 0 & 1 \end{pmatrix},
\end{align*}
which satisfies \[\rho_V(\mu_P)^w = \rho(\mu_C) \, \textrm{ and }\,\rho_V(\lambda_P) = \rho(\lambda_C)^w,\] and we compute that
\begin{align*}
\rho_V(\mu_P^{pw^2/d}\lambda_P^{q/d}) &= \epsilon^{pw^2/d} \begin{pmatrix} 1 & apw/d \\ 0 & 1 \end{pmatrix} \cdot \eta^{w(q/d)} \begin{pmatrix} 1 & bqw/d \\ 0 & 1 \end{pmatrix} \\
&= \epsilon^{pw^2/d}\eta^{w(q/d)} \begin{pmatrix} 1 & (ap+bq)w/d \\ 0 & 1 \end{pmatrix} \\
&= \epsilon^{pw^2/d}\eta^{w(q/d)} \begin{pmatrix} 1 & 0 \\ 0 & 1 \end{pmatrix}
\end{align*}
since $ap+bq=0$.  If $d$ is odd, then the leading coefficient on the right satisfies
\[ \epsilon^{pw^2/d}\eta^{w(q/d)} = \left(\epsilon^{pw^2/d}\eta^{w(q/d)}\right)^d = (\epsilon^w)^{pw} \eta^{qw} = \epsilon^{pw}\eta^{qw} = (\epsilon^p\eta^q)^w = 1. \]
Otherwise $d=\gcd(q,w^2)$ is even, hence both $q$ and $w$ must be even; but then $\eta^q = \eta^w=1$, and $\epsilon^p\eta^q=1$ implies that $\epsilon^p=1$, so that
\[ \epsilon^{pw^2/d}\eta^{w(q/d)} = (\epsilon^p)^{w^2/d} (\eta^w)^{q/d} = 1. \]
In either case we have $\rho_V(\mu_P^{pw^2/d}\lambda_P^{q/d}) = 1$, so $\rho_V$ is the desired representation.

\vspace{1em}\noindent
\underline{Case 3}: $\rho(\mu_C)$ is a $2\times 2$ Jordan block with eigenvalues $-1$, and $w$ is even.\\

As in the previous case, we can write
\begin{align*}
\rho(\mu_C) &= -\begin{pmatrix} 1 & a \\ 0 & 1 \end{pmatrix}, &
\rho(\lambda_C) &= \eta \begin{pmatrix} 1 & b \\ 0 & 1 \end{pmatrix}
\end{align*}
for some $a,b\in\C$ with $a\neq 0$ and some $\eta=\pm1$, and then $\rho(\mu_C^p\lambda_C^q)=1$ is equivalent to $(-1)^p\eta^q=1$ and $ap+bq=0$.  The problem is that now the matrix $\rho(\mu_C)$ does not have a $w$th root.  We fix this by considering the nontrivial central character
\[ \chi: \pi_1(E_C) \twoheadrightarrow H_1(E_C) \to \{\pm1\} \]
that sends $\mu_C$ to $-1$ and $\lambda_C$ to $+1$, and setting $\rho' = \chi\cdot\rho$, so that
\begin{align*}
\rho'(\mu_C) &= \begin{pmatrix} 1 & a \\ 0 & 1 \end{pmatrix}, &
\rho'(\lambda_C) &= \eta \begin{pmatrix} 1 & b \\ 0 & 1 \end{pmatrix}.
\end{align*}
We can extend $\rho'$, rather than the original $\rho$, across $\pi_1(V)$ by setting
\begin{align*}
\rho_V(\mu_P) &= \begin{pmatrix} 1 & a/w \\ 0 & 1 \end{pmatrix}, &
\rho_V(\lambda_P) &= \begin{pmatrix} 1 & bw \\ 0 & 1 \end{pmatrix},
\end{align*}
and we have $\rho_V(\mu_P)^w = \rho'(\mu_C)$ and $\rho_V(\lambda_P) = \rho'(\lambda_C)^w$ since $w$ is even.  Now 
\[ \rho_V(\mu_P^{pw^2/d}\lambda_P^{q/d}) = \begin{pmatrix} 1 & (ap+bq)w/d \\ 0 & 1 \end{pmatrix} = \begin{pmatrix} 1 & 0 \\ 0 & 1 \end{pmatrix}, \]
so $\rho_V$ is the desired extension of $\rho'$, and it has nonabelian image since $\rho'$ does.

\vspace{1em}These three cases cover all possible values of $\rho(\mu_C)$, so we conclude that $rw^2$-surgery on $K$ cannot be $\SL(2,\C)$-abelian after all.
\end{proof}

In preparation for the next result, we note that since $\SU(2)$ is a subgroup of $\SL(2,\C)$, an $\SL(2,\C)$-abelian 3-manifold is also $\SU(2)$-abelian, and hence an $\SL(2,\C)$-averse knot is also $\SU(2)$-averse. In particular, if $K$ is $\SL(2,\C)$-averse, then the set of slopes $r$ for which $S^3_r(K)$ is $\SL(2,\C)$-abelian has a unique accumulation point, which agrees with the limit slope $r(K)$ for $\SU(2)$-abelian surgeries of Theorem~\ref{thm:sz-averse}.

The combination of Lemmas \ref{lem:sl2-abelian-pinch} and \ref{lem:sl2-abelian-companion} then leads to the following, which is an analogue for $\SL(2,\C)$-averse satellites of Proposition \ref{prop:thin-satellite} on thin satellites:

\begin{corollary} \label{cor:sl2-averse-satellite}
Suppose  $K = P(C)$ is an $\SL(2,\C)$-averse satellite knot, and let $w$ be the winding number of the pattern $P$.  Then:
\begin{itemize}
\item $P(U)$ is nontrivial and $w \geq 1$,
\item $P(U)$ is $\SL(2,\C)$-averse, with $r(P(U)) = r(K)$,
\item the companion $C$ is $\SL(2,\C)$-averse, with $r(K) = w^2 r(C)$.
\end{itemize}
\end{corollary}

\begin{proof}
By definition, $K = P(C)$ has infinitely many $\SL(2,\C)$-abelian slopes \[r_1,r_2,r_3,\dots \in \Q,\] which are also $\SU(2)$-abelian slopes for $K$ and have a unique accumulation point $r(K) \in \Q$, as discussed above. Since $K$ is also $\SU(2)$-averse, Corollary \ref{cor:SUL} says that either $K$ or $\mirror{K}$ is an instanton L-space knot. Then $P(U)$ is nontrivial and $w\geq 1$ by Lemma \ref{lem:satellite-lspace}.

Lemma~\ref{lem:sl2-abelian-pinch} says that each $r_i$ is also an $\SL(2,\C)$-abelian slope for $P(U)$, so $P(U)$ is also $\SL(2,\C)$-averse with limit slope $r(P(U)) = r(K)$.

Similarly, Lemma~\ref{lem:sl2-abelian-companion} says us that each of the slopes $r_i/w^2$ is also an $\SL(2,\C)$-abelian slope for $C$, so $C$ is also $\SL(2,\C)$-averse with limit slope $r(K)/w^2$.
\end{proof}

We now prove Theorem \ref{thm:SLabeliantorus-proof}, which completes the proof of Theorem \ref{thm:SLabeliantorus}:

\begin{proof}[Proof of Theorem~\ref{thm:SLabeliantorus-proof}]
Note that $K$ is not hyperbolic, by Lemma~\ref{lem:sl2-averse-hyperbolic}. If $K$ is a torus knot then we are done, so it suffices to prove that there are no $\SL(2,\C)$-averse satellite knots. 

Suppose for a contradiction that $K=P(C)$ is an $\SL(2,\C)$-averse satellite knot with the smallest genus among all $\SL(2,\C)$-averse satellites. Corollary~\ref{cor:sl2-averse-satellite} says that $P(U)$ and $C$ are  nontrivial and $\SL(2,\C)$-averse, and that if $P$ has winding number $w$, then $w\geq 1$ and 
\[ r(P(U)) = r(K) = w^2 r(C). \]
Then from the relation
\[ g(K) = g(P(U)) + w\cdot g(C), \]
we see that $P(U)$ and $C$ have  genera strictly less than that of $K$, and are thus  torus knots. 

Now we repeat the proof of Theorem~\ref{thm:thin-torus-proof}: if we write $P(U) = T_{a,b}$ and $C = T_{c,d}$, then the above relation between their limit slopes becomes
\[ ab = r(K) = cd w^2, \]
so $w^2$ divides $ab$.  Proposition \ref{prop:torus-satellite-non-lspace} then implies that neither $K$ nor $\mirror{K}$ is an instanton L-space knot. But this yields a contradiction: since $K$ is $\SL(2,\C)$-averse, it is also $\SU(2)$-averse, which implies by Corollary \ref{cor:SUL} that either $K$ or  $\mirror{K}$ is an instanton L-space knot. 
\end{proof}

%

\bibliographystyle{myalpha}
\bibliography{References}

\end{document}